\documentclass[a4paper,11pt]{amsart}
\usepackage[OT2,OT1]{fontenc}
\usepackage[latin1]{inputenc}
\usepackage{amsmath}
\usepackage{amsthm}
\usepackage{amscd}
\usepackage{amssymb}
\usepackage{pb-diagram, pb-xy}
\usepackage[all]{xy}
\usepackage{hyperref}

\newtheorem{thm}{Theorem}[section]
\newtheorem*{thm*}{Main Theorem}
\newtheorem{cor}[thm]{Corollary}
\newtheorem*{cor*}{Corollary}
\newtheorem{lem}[thm]{Lemma}
\newtheorem{klem}[thm]{Key Lemma}
\newtheorem{prop}[thm]{Proposition}
\theoremstyle{definition}
\newtheorem{defn}[thm]{Definition}
\newtheorem{rem}[thm]{Remark}

\newtheorem{para}[thm]{--}

\newcommand\cyr{%
\renewcommand\rmdefault{wncyr}%
\renewcommand\sfdefault{wncyss}%
\renewcommand\encodingdefault{OT2}%
\normalfont\selectfont}
\DeclareTextFontCommand{\textcyr}{\cyr}

\newcommand{\angl}[1]{\left\langle #1\right\rangle}

\newcommand{\tq}{\: | \:}
\newcommand{\pt}{\mbox{ for all }}
\newcommand{\qqet}{\qquad\mbox{and}\qquad}

\DeclareMathOperator{\End}{End}% Endomorphisme
\DeclareMathOperator{\Gal}{Gal}% Gr de Galois
\DeclareMathOperator{\GL}{GL}% Gr lineaire
\DeclareMathOperator{\Hom}{Hom}% Homomorphisme
\DeclareMathOperator{\id}{id} % Identity
\DeclareMathOperator{\im}{im} % Image
\DeclareMathOperator{\Lie}{Lie} % 
\DeclareMathOperator{\red}{red}% Rang
\DeclareMathOperator{\rk}{rank}% Rang
\DeclareMathOperator{\spec}{spec}% Spectre
\newcommand{\hotimes}{\:\widehat\otimes\:}
\newcommand{\Tell}{\textup{T}_\ell}
\newcommand{\TZ}{\textup{T}_{\mathbb Z}}
\newcommand{\Vell}{\textup{V}_{\!\ell}}
\newcommand{\Vnot}{\textup{V}_{\!0}}

\newcommand{\IC}{\mathbb{C}}

\newcommand{\IL}{\mathbb{L}}

\newcommand{\IP}{\mathbb{P}}
\newcommand{\IQ}{\mathbb{Q}}
\newcommand{\IR}{\mathbb{R}}
\newcommand{\IZ}{\mathbb{Z}}

\newcommand{\cO}{\mathcal O}

\newcommand{\cHom}{\mathcal Hom}

\newcommand{\fh}{\mathfrak h}
\newcommand{\fl}{\mathfrak l}

\newcommand{\fp}{\mathfrak p}

\renewcommand{\to}{\xrightarrow{\quad}}
\renewcommand{\mapsto}{\longmapsto}

\date{\today}

\title{Detecting linear dependence on a simple abelian variety}%
\author{Peter Jossen}%

\begin{document}

\begin{abstract}
Let $A$ be a geometrically simple abelian variety over a number field $k$, let $X$ be a subgroup of $A(k)$ and let $P\in A(k)$ be a rational point. We prove that if $P$ belongs to $X$ modulo almost all primes of $k$ then $P$ already belongs to $X$.
\end{abstract}

\maketitle%

\section*{Introduction}

\begin{par}
Let $A$ be an abelian variety over a number field $k$, let $X$ be a subgroup of the Mordell--Weil group $A(k)$ and let $P \in A(k)$ be a rational point. We want to ``decide'' whether $P$ belongs to $X$ or not. To do so, we choose a model of $A$ over an open subscheme $U$ of $\spec\cO_k$, where $\cO_k$ denotes the ring of integers of $k$. Because $A$ is proper, $P$ and all points in $X$ extend to $U$--points. For closed points $\fp\in U$ we can consider the reduction map
$$\red_\fp: A(U) \to A(\kappa_\fp)$$
where $\kappa_\fp := \cO_k/\fp$ denotes the residue field at $\fp$. A necessary condition for $P$ belonging to $X$ is then that for all closed points $\fp \in U$ the reduction of $P$ modulo $\fp$ belongs to the reduction of $X$ modulo $\fp$. Wojciech Gajda asked in 2002 whether this condition is also sufficient. This problem was named the problem of \emph{detecting linear dependence}.
\end{par}

\begin{par}
In a joint work with Antonella Perucca (\cite{JPer09}) we have shown that the answer to Gajda's question is negative in general by giving an explicit counterexample (Banaszak and Kras\'on have found independently such a counterexample). The abelian variety in our counterexample is a power of an elliptic curve. Our main result in this note is:
\end{par}

\vspace{4mm}
\begin{par}{\bf Main Theorem.}\emph{
Let $A$ be a geometrically simple abelian variety over a number field $k$, let $X$ be a subgroup of $A(k)$ and let $P\in A(k)$ be a rational point. If the set of places $\fp$ of $k$ for which $\red_\fp(P)$ belongs to $\red_\fp(X)$ has natural density 1, then $P$ belongs to $X$.}
\end{par}

\vspace{4mm}
\begin{par}
By saying that $A$ is \emph{geometrically simple} we mean that $A$ has no other abelian subvariety other than $0$ and itself defined over an algebraic closure $\overline k$ of $k$. The statement of the theorem is new even in the case where $A$ is an elliptic curve. However, many partial results in this direction have already been obtained, let us mention a few of them. The earliest result on this problem is due to Schinzel (\cite{Schin75}), who showed the analogue of our Main Theorem for the multiplicative group in place of an abelian variety. Weston has shown that for an abelian variety with a commutative endomorphism ring the statement of our theorem holds up to a torsion ambiguity (\cite{West03}), and Kowalski has shown the statement of our theorem to hold for an elliptic curve and a cyclic subgroup (\cite{Kowa03}). Banaszak, Gajda, G\'ornisiewicz and Kraso\'n have proven similar statements under various technical assumptions on the abelian variety and the subgroup (\cite{Bana05, Gajd08, Bana09}), and Perucca has some similar results for products of tori and abelian varieties (\cite{Peru08}).
\end{par}

\vspace{2mm}
\begin{par}
Here is a quick overview on the main ideas of the proof. Let $U$ be an open subscheme of $\spec\cO_k$, where $\cO_k$ is the ring of integers of the number field $k$. A 1--motive over $U$ is a morphism of fppf sheaves
$$M=[u:Y\to G]$$
over $U$ where $Y$ is \'etale locally constant, locally isomorphic to a finitely generated free group, and where $G$ is a semiabelian scheme over $U$. By a semiabelian scheme over $U$ we understand in this paper an extension over $U$ of an abelian scheme scheme by a torus. In the case $Y$ is constant defined by a finitely generated free group which we still denote by $Y$, morphisms of fppf--sheaves $Y\to G$ are the same as homomorphisms of groups $Y\to G(U)$. Given a semiabelian scheme $G$ over $U$ and a finitely generated subgroup $X$ of $G(U)$ we can choose a 1--motive $[Y\to G]$ over $U$ where $Y$ is a constant sheaf defined by a finitely generated free group, such that $u(Y)=X$. In the case $X$ is torsion free on can just take $Y=X$ and for $u$ the inclusion.\\
With any 1--motive $M$ over $U$ and prime number $\ell$ invertible on $U$ is associated a locally constant $\ell$--adic sheaf $\Tell M$ on $U$, which can also be viewed as a finitely generated free $\IZ_\ell$--module equipped with a continuous action of the absolute Galois group of $k$ which is unramified in $U$. For a set $S$ of closed points of $U$ of density 1 we consider the group
$$H^1_S(U,\Tell M) := \ker\bigg(H^1(U,\Tell M)\to \prod_{\fp\in S}H^1(\kappa_\fp,\Tell M)\bigg)$$
where $\kappa_\fp = \cO_k/\fp$ denotes the residue field at $\fp$. Using Kummer theory we will show that the vanishing of the groups $H^1_S(U,\Tell M)$ for all $\ell$ is the obstruction for the local--global principle of the Main Theorem to hold. As observed by Serre and Tate it is essentially a consequence of Chebotarev's Density Theorem that the group $H^1_S(U,\Tell M)$ is isomorphic to the group
$$H^1_\ast(L^M,\Tell M) := \ker\bigg(H^1(L^M,\Tell M)\to \prod_{C\leq L^M}H^1(C,\Tell M)\bigg)$$
where $L^M$ denotes the image of the Galois group $\Gal(\overline k|k)$ in the group of automorphisms of $\Tell M$ and where the product ranges over all subgroups $C$ of $L^M$ topologically generated by one element. In the case where $G$ is an abelian variety we will determine the group $L^M$ up to comensurability, and modulo the Mumford--Tate conjecture. This will allow us then, in the case where $A$ is geometrically simple, to gain sufficient control on $H^1_\ast(L^M,\Tell M)$ in order to prove the Main Theorem.
\end{par}

\vspace{4mm}%
\begin{par}{\bf Acknowledgments:}
Large parts of this article are taken from my Ph.D. thesis directed by Tam\'as Szamuely. I wish to thank him for his help, encouragement and support during this work. Many thanks go to Antonella Perucca who considerably helped to simplify some of the arguments. I am greatful to G. Banaszak and W. Gajda for very useful correspondance and to G. Banaszak and P. Kraso\'n for pointing out a mistake in an earlier version of this text. I acknowledge financial support provided by the DFG-Forschergruppe "Algebraische Zykel und L-Funktionen", Regensburg.
\end{par}

\newpage
\tableofcontents

%===============================================================================
\vspace{10mm}%==================================================================
\section{On 1--motives and Galois representations}%=============================
%===============================================================================

\begin{par}
In this section I recall what 1--motives are and how to attach $\ell$--adic Galois representations to them. Then I show how these representations are linked with the local--global problem of detecting linear dependence.
\end{par}

\vspace{4mm}
\begin{para}\label{Par:IntroMotive}
Let $S$ be a noetherian regular scheme. A \emph{1--motive} $M$ over $S$ is (\cite{Deli74}, Section 10) a two--term complex of fppf--sheaves over $S$, concentrated in degrees $-1$ and $0$ 
$$M := [Y \xrightarrow{\:\:u\:\:} G]$$
where $Y$ is \'etale locally isomorphic to a finitely generated free $\IZ$--module and where $G$ is representable by a semiabelian scheme over $S$. A \emph{morphism of 1--motives} is a morphism of complexes of fppf--sheaves. One can view $M$ as an object of the derived category of fppf--sheaves on $S$. Applying the derived global section functor $\IR\Gamma(S,-)$ and taking homology yields the flat cohomology groups $H^i(S,M)$. There is a long exact sequence relating the cohomology of $G$ and $Y$ with that of $M$ starting with
$$0\to H^{-1}(S,M) \to H^0(S,Y) \to H^0(S,G) \to H^0(S,M) \to H^1(S,Y) \to \cdots$$
One can also view $M$ as an object of the derived category of \'etale sheaves and obtain \'etale cohomology groups. However, since $G$ and $Y$ are both smooth over $S$, these are canonically isomorphic.
\end{para}

\vspace{4mm}
\begin{para}{\bf Notation:}
For a commutative group $C$, a prime number $\ell$ and an integer $i\geq 0$, we introduce the following notation: $C[\ell^i]$ denotes the group of elements of $C$ of order $\ell^i$, and $C[\ell^\infty]$ denotes the group of elements of $C$ of order any power of $\ell$. We write
$$C\hotimes\IZ_\ell := \lim_{i\geq 0} C/\ell^iC \qqet \Tell C := \lim_{i\geq 0} C[\ell^i]$$
for the $\ell$--adic completion and the $\ell$--adic Tate module of $C$. These groups have a natural $\IZ_\ell$--module structure. There is a canonical morphism $C \to C \hotimes \IZ_\ell$ whose kernel is the intersection of the groups $\ell^iC$ over $i \geq 0$. Remark that if $C$ is finitely generated, we may identify $C\hotimes\IZ_\ell \cong C \otimes_\IZ \IZ_\ell$.
\end{para}

\vspace{4mm}
\begin{para}\label{Par:IntroTateMod}
Following Deligne (\emph{loc.cit.}) we now construct the $\ell$--adic Tate module associated with (or $\ell$--adic realisation of) a 1--motive $M=[u:Y \to G]$ over $S$, where $\ell$ is any prime number invertible on $S$. We shall consider the derived tensor product $M\otimes^\IL\IZ/\ell^i\IZ$, or alternatively (that amounts to the same) the cone of the multiplication--by--$\ell^i$ map on the complex $M$. The homology of $M \otimes^\IL \IZ/\ell^i\IZ$ is concentrated in degree $-1$ because $Y$ is torsion free and $G$ is divisible as a sheaf. The homology group
$$T_{\IZ/\ell^i\IZ}(M) := H^{-1}(M \otimes^\IL \IZ/\ell^i\IZ)$$
is a finite flat group scheme over $S$ annihilated by $\ell^i$, and because we suppose that $\ell$ is invertible on $S$ it is locally constant. We have a natural morphism $T_{\IZ/\ell^{i+1}\IZ}(M) \to T_{\IZ/\ell^i\IZ}(M)$ induced by the map $\IZ/\ell^{i+1}\IZ\to \IZ/\ell^i\IZ$ for all $i\geq0$. The formal limit with respect to these maps
$$\Tell M := \lim_{i\geq 0} T_{\IZ/\ell^i\IZ}(M)$$
is a locally constant $\ell$--adic sheaf on $S$, called \emph{the $\ell$--adic Tate module of $M$}. This construction is functorial in $M$ so we look at $\Tell(-)$ as being a functor from the category of 1--motives over $S$ to the category of $\ell$--adic sheaves over $S$. The cohomology of $\Tell M$ over $S$ is then defined accordingly as
$$H^r(S,\Tell M) := \lim_{i\geq 0}H^{r-1}(S, M \otimes^\IL \IZ/\ell^i\IZ)$$
These cohomology groups have a natural $\IZ_\ell$--module structure. There are natural short exact sequences as follows. The exact ``Kummer" triangle $M \to M \to M\otimes^\IL \IZ/\ell^i\IZ$ induces a long exact sequence of cohomology groups from where we can cut out the piece
$$0\to H^{r-1}(S,M)/\ell^i H^r(S,M) \to H^{r-1}(S,M\otimes^\IL \IZ/\ell^i\IZ)\to H^r(S,M)[\ell^i] \to 0$$
Taking limits over $i$ and observing that the left hand limit system satisfies the Mittag--Leffler condition, we find a short exact sequence of $\IZ_\ell$--modules
$$0\to H^{r-1}(S,M)\hotimes\IZ_\ell \to H^r(S,\Tell M)\to \Tell H^r(S,M) \to 0$$
Naturality in $M$ and $S$ is clear from the construction.
\end{para}

\vspace{4mm}
\begin{para}\label{Par:FundamentalSetup}
For the rest of this section we fix a number field $k$ with algebraic closure $\overline k$ and absolute Galois group $\Gamma := \Gal(\overline k|k)$, a nonempty open subscheme $U$ of $\spec\cO_k$ where $\cO_k$ denotes the ring of integers of $k$, and a prime number $\ell$ invertible on $U$. We write $k_U$ for the maximal subextension of $\overline k|k$ unramified in $U$, and set $\Gamma_U := \Gal(k_U|k)$. In other words, $\Gamma_U = \pi_1(U,u)$ is the \' etale fundamental group of $U$ with respect to the base point $u=\spec\overline k$.
\end{para}

\vspace{4mm}
\begin{para}\label{Par:GrothEquivalenceY}
By  Grothendieck's theory of the fundamental group (see for example \cite{Szam09}, Theorem 5.4.2), there is an equivalence of categories
$$\bigg\{\begin{minipage}{40mm}
locally constant $\IZ$--con-\\structible sheaves on $U$
\end{minipage}
\bigg\} \longleftrightarrow \bigg\{\begin{minipage}{35mm}
finitely generated\\ discrete $\Gamma_U$--modules
\end{minipage}
\bigg\}$$
given by the functor that sends such a sheaf $F$ on $U$ to the $\Gamma_U$--module $F(\overline k)$. In particular, to give a locally constant sheaf $Y$ locally isomorphic to a finitely generated free group is the same, via this equivalence of categories, as to give a finitely generated free group $Y$ together with a continuous action of $\Gamma_U$. Continuity means that the stabiliser of $Y$ in $\Gamma_U$ is an open subgroup of finite index. As a consequence, a 1--motive over $U$ is given by the following data: A finitely generated free group $Y$ together with a continuous action of $\Gamma_U$, a semiabelian scheme $G$ over $U$ and a morphism of $\Gamma_U$--modules $u:Y\to G(k_U)$.
\end{para}

\vspace{4mm}
\begin{para}\label{Par:GrothEquivalenceEll}
The equivalence of categories given in \ref{Par:GrothEquivalenceY} also explains why $\ell$--adic sheaves on $U$ are essentially the same as $\ell$--adic representations of $k$ unramified in $U$. Indeed, this equivalence of categories induces an equivalence
$$\hspace{12mm}\bigg\{\begin{minipage}{40mm}
locally constant $\ell$--adic\\ sheaves on $U$
\end{minipage}
\bigg\} \longleftrightarrow \bigg\{\begin{minipage}{50mm}
finitely generated $\IZ_\ell$--modules \\with continuous $\Gamma_U$--action
\end{minipage}
\bigg\}$$
given by the functor that sends a locally constant $\ell$--adic sheaf on $U$, given by a formal limit system $(T_i)_{i=0}^\infty$ to the $\IZ_\ell$--module $\lim T_i(\overline k)$. A quasi inverse to this functor is can be defined as follows: Given a finitely generated $\IZ_\ell$--module $T$ with continuous $\Gamma_U$--action, one associates with it the formal limit system $(T_i)_{i=0}^\infty$ where $T_i$ is the locally constant sheaf on $U$ corresponding to the finite $\Gamma_U$--module $T/\ell^iT$. 
\end{para}

\vspace{4mm}
\begin{para}\label{Par:TellMExplicit}
Using the equivalence of categories introduced in \ref{Par:GrothEquivalenceEll}, we can give an explicit description of the Tate module of a 1--motive $M=[u:Y\to G]$ over $U$ in terms of Galois representations. For all $i\geq 0$ we have finite Galois modules
$$T_{\IZ/\ell^i\IZ}(M)(\overline k) \cong \frac{\{(y,P)\in Y\times G(\overline k)\tq u(y) = \ell^iP\}}{\{(\ell^iy,u(y))\tq y \in Y\}}$$
which are unramified in $U$. The limit over $i$ of these finite Galois modules is then the Tate module of $M$ seen as a Galois module. Explicitly, an element $x \in \Tell M$ is given by a sequence $(y_i,P_i)_{i=0}^\infty$ where the $y_i$'s are elements of $Y$, the $P_i$'s are elements of $G(\overline k)$, and where it is required that
$$u(y_i) =  \ell^iP_i \qqet  \ell P_i - P_{i-1} = u(z_i) \qqet y_i -y_{i-1} = \ell^{i-1}z_i$$
for some elements $z_i \in Y$. Two sequences $(y_i,P_i)_{i=0}^\infty$ and $(y_i',P_i')_{i=0}^\infty$ represent the same element if and only if for each $i \geq 0$, there exists a $z_i \in Y$ such that $\ell^iz_i = y_i-y_i'$ and $u(z_i) = P_i-P_i'$.
\end{para}

\vspace{4mm}
\begin{prop}
Let $T = (T_i)_{i=0}^\infty$ be a locally constant $\ell$--adic sheaf on $U$ corresponding via the above equivalence to a $\IZ_\ell$--module with continuous $\Gamma_U$--action (also denoted by $T$). For $r = 0,1$, the natural maps
$$H^r(\Gamma_U, T) \to H^r(U,T)$$
are isomorphisms, where $H^r(\Gamma_U, T)$ is defined by means of continuous cocycles.
\end{prop}

\begin{proof}
From Proposition II.2.9 of \cite{MilneADT} we know that if $F$ is a finite locally constant sheaf of order a power of $\ell$ on $U$, then we have canonical isomorphisms $H^r(U,F) \cong H^r(\Gamma_U, F)$ for all $r \geq 0$. Cohomology of $\ell$--adic sheaves over $U$ commutes with limits by definition. It remains to prove that if $T$ is a finitely generated $\IZ_\ell$--module with $\Gamma_U$--action, then the natural map
$$H^r(\Gamma_U, T) \to \lim_{i\geq 0} H^r(\Gamma_U, T/\ell^i T)$$
is an isomorphism for $r = 0,1$. For $r=0$ this is trivial, and for $r=1$ this follows from the well known fact that continuous $H^1$ commutes with limits of compact modules (see Proposition 7 of \cite{Serre1}).
\end{proof}

\vspace{4mm}
\begin{prop}\label{Pro:TateModFixed}
Let $M = [u:Y\to G]$ be a 1--motive over $k$. There is a canonical isomorphism $(\Tell M)^\Gamma \cong \ker(Y^\Gamma \to G(k))\otimes\IZ_\ell$.
\end{prop}

\begin{proof}
Let $U$ be an open subscheme of $\spec\cO_k$ such that there is a model of $M$ over $U$, which we still denote by $M$. We have a short exact sequence
$$0\to H^{-1}(U,M) \hotimes \IZ_\ell \to H^0(U,\Tell M) \to \Tell H^0(U,M)\to 0$$
as introduced in \ref{Par:IntroTateMod}. The group $H^0(U,M)$ is finitely generated (that follows by d\'evissage from the Mordell--Weil theorem, Dirichlet's unit theorem and the finiteness of $H^1(U,Y)$, see \cite{Hara05}, Lemma 3.2) hence $\Tell H^0(U,M)$ is trivial. We remain with an isomorphism
$$H^{-1}(U,M) \otimes \IZ_\ell \to H^0(U,\Tell M)$$
but now, observe that $H^{-1}(U,M) = \ker(Y^\Gamma \to G(k))$ and that $H^0(U,\Tell M) \cong (\Tell M)^{\Gamma}$.
\end{proof}

\vspace{4mm}
\begin{defn}
Let $T$ be an $\ell$--adic sheaf on $U$ and let $S$ be a set of closed points of $U$. For each $\fp\in S$ let $\kappa_\fp$ be the residue field at $\fp$ and denote still by $T$ the pull--back of $T$ to $\spec \kappa_\fp$. We define
$$H^1_S(U,T) := \ker\bigg(H^1(U,T)\to \prod_{\fp\in S}H^1(\kappa_\fp,T)\bigg)$$
Alternatively, in terms of Galois cohomology, let $\Gamma_U$ be the Galois group of the maximal extension of $k$ unramified in $U$ and let $D_\fp$ be a decomposition group of $\fp$ in $\Gamma_U$. For every finitely generated free $\IZ_\ell$--module with continuous $\Gamma_U$--action $T$ we define
$$H^1_S(\Gamma_U,T) := \ker\bigg(H^1(\Gamma_U,T)\to \prod_{\fp\in S}H^1(D_\fp,T)\bigg)$$
Observe that the choice of decomposition groups $D_\fp$ is unimportant since all decomposition groups over $\fp$ are conjugate, and a cocycle $c: \Gamma_U \to T$ restricts to a coboundary on $D_\fp$ if and only if it restricts to a coboundary on some conjugate of $D_\fp$.
\end{defn}

\vspace{4mm}
\begin{prop}\label{Pro:InjectionOtimesEll}
Let $k$ be a number field, let $G$ be a semiabelian scheme over $U$ and let $X$ be a subgroup of $G(U)$. Let $S$ be a set of closed points of $U$ of density 1 and write 
$$\overline X := \{P\in G(U) \tq \red_\fp(P)\in \red_\fp(X) \pt \fp\in S\}$$
Let $M = [u:Y\to G]$ be a 1--motive over $U$ where $Y$ is constant and such that $u(Y)$ is equal to $X$. For every prime number $\ell$ invertible on $U$ there exists a canonical, $\IZ_\ell$--linear injection $(\overline X/X) \otimes \IZ_\ell \to H^1_S(\Gamma_U, \Tell M)$.
\end{prop}

\begin{proof}
We have chosen a 1--motive $M=[u:Y\to G]$ over $U$ with constant $Y$, such that the image of $Y \to G(U)$ is $X$. The image of $Y \to G(\kappa_\fp)$ is then $X_\fp$, the reduction of $X$ modulo $\fp$. So, if $\fp$ is any element of $S$, then every point $P\in \overline X$ maps to zero in $H^0(\kappa_\fp, M)$  in the following diagram with exact rows
$$\begin{diagram}
\setlength{\dgARROWLENGTH}{4mm}
\node{\cdots} \arrow{e}\node{Y}\arrow{s,=}\arrow{e,t}{u_U}\node{G(U)}\arrow{s}\arrow{e}\node{H^0(U,M)}\arrow{s} \arrow{e}\node{0 = H^1(U,Y)\hspace{-7mm}}\\
\node{\cdots} \arrow{e}\node{Y}\arrow{e,t}{u_\fp}\node{G(\kappa_\fp)}\arrow{e}\node{H^0(\kappa_\fp,M)} \arrow{e}\node{0= H^1(\kappa_\fp,Y)\hspace{-7mm}}
\end{diagram}$$
Denote by $[P]$ the class of $P \in \overline X$ in $H^0(U,M) \cong G(U)/X$. We have seen that $[P]\otimes 1$ belongs to the kernel of the map $\alpha_\ell$ in the diagram
$$\begin{diagram}
\setlength{\dgARROWLENGTH}{4mm}
\node{0}\arrow{e}\node{H^0(U,M)\otimes \IZ_\ell} \arrow{s,l}{\alpha_\ell}\arrow{e} \node{H^1(U,\Tell M)}\arrow{s,l}{\beta_\ell} \arrow{e}\node{\Tell H^1(U,M)}\arrow{s} \arrow{e}\node{0}\\
\node{0}\arrow{e}\node{\!\!\prod \!H^0(\kappa_\fp,\!M)\!\otimes\! \IZ_\ell\!\!} \arrow{e} \node{\!\!\prod\!H^1(\kappa_\fp,\!\Tell M)\!\!} \arrow{e}\node{\!\!\prod\Tell H^1(\kappa_\fp,M)} \arrow{e}\node{0}
\end{diagram}$$
The rows are those introduced in \ref{Par:IntroTateMod} and the products range over $\fp\in S$. The $\ell$--adic completions are here just ordinary tensor products because the involved groups are all finitely generated (\cite{Hara05}, Lemma 3.2). We have natural injections
$$(\overline X/X)\otimes \IZ_\ell \subseteq \ker\alpha_\ell \subseteq \ker\beta_\ell = H^1_S(U,\Tell M)$$
hence the claim.
\end{proof}

\vspace{4mm}
\begin{rem}
The injection whose existence we claim in Proposition \ref{Pro:InjectionOtimesEll} is explicitly given as follows. Let $P$ be an element of $\overline X$, and denote by $[P]$ its class in $\overline X/X$. Choose a sequence of points $(P_i)_{i=0}^\infty$ in $G(\overline k)$ such that $P_0=P$ and such that $\ell P_{i+1}=P_i$ for all $i\geq 0$. The image of $[P]\otimes 1$ in $H^1_\ast(\Gamma_U, \Tell M)$ via the injection under consideration is the class of the cocycle $c_P:\Gamma_U\to \Tell M$ given by
$$c_P: \sigma \mapsto (\sigma P_i - P_i)_{i=0}^\infty$$
This makes sense since indeed each $\sigma P_i - P_i$ is a point in $G(\overline k)$ of order $\ell^i$, and together these points form a compatible system representing an element of the Tate module $\Tell G$, which is a submodule of $\Tell M$.
\end{rem}

\vspace{4mm}
\begin{rem}
Let $G$ be any semiabelian variety over $k$, let $X$ be a \emph{finitely generated} subgroup of $G(k)$ and let $\ell$ be any prime number. It is always possible to find an open subscheme $U$ of $\spec\cO_k$ such that $G$ has a model over $U$, such that all points in $X$ extend to $U$--points, and such that $\ell$ is invertible on $U$. Also observe that $G(U)$ is finitely generated, as a direct consequence of the Mordell--Weil theorem and Dirichlet's unit theorem.
\end{rem}

\vspace{4mm}
\begin{para}
For a 1--motive $M$ over $U$ we may regard the $\ell$--adic sheaf $\Tell M$ as a finitely generated free $\IZ_\ell$--module with continuous $\Gamma_U$--action, as we have explained, $\Gamma_U$ being the Galois group of the maximal extension of $k$ unramified in $U$. The following definition goes back to an idea of Tate and Serre: For a Hausdorff topological group $\Gamma$ and a continuous $\Gamma$--module $T$ we write
$$H^1_\ast(\Gamma,T) := \ker\!\bigg(H^1(\Gamma,T) \to \!\!\prod_{C \leq \Gamma}\!\!H^1(C,T)\bigg)$$
the product running over monogenous subgroups $C$ of $\Gamma$, cohomology being defined by means of continuous cochains. A subgroup of a topological group is called \emph{monogeneous} if is topologically generated by one element, that is, if it is the closure of a subgroup generated by one element. The following two propositions (\cite{Serre1}, Proposition 8 and Proposition 6) explain why the group $H^1_\ast(\Gamma_U, \Tell M)$ is interesting.
\end{para}

\vspace{4mm}
\begin{prop}\label{Pro:SerreGabber}
Let $T$ be a finitely generated $\IZ_\ell$--module with a continuous $\Gamma_U$--action and let $S$ be a set of closed points of $U$ of density 1. The subgroups $H^1_S(\Gamma_U,T)$ and $H^1_\ast(\Gamma_U,T)$ of $H^1(\Gamma_U,T)$ are equal.
\end{prop}

\begin{proof}
It is enough to show that the proposition holds for finite Galois modules of order a power of $\ell$. Indeed, $T$ can be written as a limit of such and the general case follows then because $H^1$ commutes with limits of finite modules, and formation of limits is left exact and commutes with products.\\
So let $F$ be a finite $\Gamma_U$ module of order a power of $\ell$. Let $c: \Gamma_U \to F$ be a continuous cocycle representing an element of $H^1_S(\Gamma_U,F)$ and let $\sigma$ be an element of $\Gamma_U$. We have to show that the restriction of $c$ to the monogeneous subgroup of $\Gamma_U$ generated by $\sigma$ is a coboundary, that is, we have to show that there exists an element $x\in F$ such that $c(\sigma)=\sigma x-x$.\\
Because $F$ is finite there exists an open subgroup $N$ of $\Gamma_U$ on which $c$ is zero. We may suppose that $N$ is normal and acts trivially on $F$. Denote by $\sigma_N$ the class of $\sigma$ in $\Gamma_U/N$. By Chebotarev's density theorem (see for example \cite{Neukirch99} Theorem 13.4), there exists a valuation $v$ of $k$ corresponding to an element $\fp \in S$ and an extension $w$ of $v$ to $k_U$ such that decomposition group of $w$ in $\Gamma_U/N$ equals the group generated by $\sigma_N$. Since the restriction of $c$ to the decomposition group $D_w \subseteq \Gamma_U$ is a coboundary, there exists a $x\in F$ such that
$$c(\tau) = \tau x-x \qquad \pt \tau\in D_w$$
As $N$ acts trivially on $F$, the same holds for all $\tau\in D_wN$, and in particular for $\tau = \sigma$. This shows that $H^1_S(\Gamma_U,F)$ is contained in $H^1_\ast(\Gamma_U,F)$. That $H^1_\ast(\Gamma_U,F)$ is contained in $H^1_S(\Gamma_U,F)$ is clear, since every decomposition group in $\Gamma_U$ corresponding to a place in $S$ is monogenous, topologically generated by the Frobenius element.
\end{proof}

\vspace{4mm}
\begin{prop}\label{Pro:H1astFundamental}
Let $\Gamma$ be a Hausdorff topological group and let $T$ be a a continuous $\Gamma$--module. Let $N$ be a normal closed subgroup of $\Gamma$ acting trivially on $T$. The inflation map $H^1(\Gamma/N,T)\to H^1(\Gamma, T)$ induces an isomorphism $H^1_\ast(\Gamma/N,T)\cong H^1_\ast(\Gamma, T)$.
\end{prop}

\begin{proof}
This is straightforward to check, see \cite{Serre1}, Proposition 6.
\end{proof}

\vspace{4mm}
\begin{para}
This has the following interesting consequence: Let us denote by $L^M$ be the image of $\Gamma_U$ in $\GL(\Tell M)$. Together, Propositions \ref{Pro:SerreGabber} and \ref{Pro:H1astFundamental} yield a canonical isomorphism
$$H^1_\ast(L^M, \Tell M) \cong H^1_S(\Gamma_U, \Tell M)$$
Since $\Gamma_U$ is compact this image $L^M$ is a closed subgroup of $\GL(\Tell M)$, hence has the structure of an $\ell$--adic Lie group (\cite{BourLie}, Ch.III, \S 2, no.2, th\'eor\`eme 2). We therefore can apply the machinery of $\ell$--adic Lie theory, and if we have sufficient knowledge of this Lie group and its Lie algebra, there might be a chance of effectively computing $H^1_\ast(L^M, \Tell M)$, hence $H^1_\ast(\Gamma_U, \Tell M)$. In the next section we will determine $L^M$ as far as we need.
\end{para}

%===============================================================================
\vspace{14mm}%==================================================================
\section{The image of Galois}%==================================================
%===============================================================================

\begin{par}
Let $k$ be a number field contained in $\IC$ and let $M = [Y\to G]$ be a 1--motive over $k$. To $M$ and every prime number $\ell$ we have associated a finitely generated free $\IZ_\ell$--module with a continuous Galois action $\Tell M$. We define
$$\Vell M := \Tell M \otimes_{\IZ_\ell}\IQ_\ell$$
so $\Vell M$ is a finite dimensional $\IQ_\ell$--vector space, and we have a continuous group homomorphism
$$\rho_\ell:\Gal(\overline k|k) \to \GL(\Vell M)$$
We have already noted that the image $L^M$ of the map $\rho_\ell$ is a compact $\ell$--adic Lie subgroup of $\GL(\Vell M)$. We write $\fl^M \subseteq \End(\Vell M)$ corresponding Lie algebra. The aim of this section is to say something halfway useful about the Lie algebra $\fl^M$. We restrict ourselves to 1--motives of the form $M=[Y\to A]$ where $A$ is an abelian variety (rather than a semiabelian variety). 
\end{par}

\vspace{4mm}
\begin{defn}\label{Def:ConstructionVnotM}
Let $M=[Y\to A]$ be a 1--motive over $k$ where $A$ is an abelian variety. We write $\TZ(M)$ for the the pull--back of $Y$ and $\Lie A(\IC)$ over $A(\IC)$ (in the category of commutative groups) and define $\Vnot M := \TZ(M)\otimes\IQ$.
\end{defn}

\vspace{4mm}
\begin{para}\label{Par:ConstructionVnotM}
The $\IQ$--vector space $\Vnot M$ has finite dimension $2\dim A+ \rk Y$, and the $\IC$--vector space $\Vnot M \otimes\IC$ carries a Hodge decomposition of type $(0,0),(0,1),(1,0)$ (\cite{Deli74}, Lemme 10.1.3.2). Hence $\Vnot M$ is an rational mixed Hodge structure. It is called the \emph{rational Hodge realisation} of $M$. 
By construction we have a short exact sequence 
$$0\to \Vnot A \to \Vnot M \to Y\otimes \IQ \to 0$$
and there is a canonical lift $\natural: \ker u\otimes\IQ \to \Vnot M$ of the inclusion of $\ker u\otimes\IQ \subseteq Y\otimes\IQ$. The next proposition is Deligne's construction 10.1.6 of \emph{loc.cit}.
\end{para}

\vspace{4mm}
\begin{prop}\label{Pro:TateModIsRationalM}
For every prime number $\ell$ there is a canonical and natural isomorphism of $\IQ_\ell$--vector spaces $\Vnot M \otimes \IQ_\ell \cong \Vell M$.
\end{prop}

\begin{proof}
We show that there is even a natural isomorphism of $\IZ_\ell$--modules $\TZ(M)\otimes \IZ_\ell \cong \Tell M$. To do so, we must show that there are natural and compatible isomorphisms of finite groups
$$\ell^{-i}\TZ(M)/\TZ(M) \xrightarrow{\:\:\cong\:\:}T_{\IZ/\ell^i\IZ}(M)(\overline k)$$
Indeed, elements of $\TZ(M)$ are pairs $(y,x)\in Y\times \Lie A(\IC)$ such that $u(y) = \exp(x)$. Hence elements of $\ell^{-i}\Lambda_M$ are pairs $(y,x)\in Y\times\Lie A(\IC)$ such that $\ell^iu(y) = \ell^i\exp(x)$. Using the expression for $T_{\IZ/\ell^i\IZ}(M)(\overline k)$ introduced in \ref{Par:TellMExplicit}, we must show that there is are natural isomorphisms
$$\frac{\{(y,x)\in Y\times\Lie A(\IC) \tq \ell^iu(y) = \ell^i\exp(x)\}}{\{(y,x)\in Y\times\Lie A(\IC)\tq u(y) = \exp(x)\}} \xrightarrow{\:\:\cong\:\:} \frac{\{(y,P)\in Y \times A(\overline k)\tq u(y) = \ell^iP\}}{\{(\ell^iy,u(y))\tq y \in Y(\overline k)\}}$$
The isomorphisms we are looking for are given by $(y,x) \mapsto (\ell^iy, \exp(x))$. Compatibility is straightforward to check and naturality is clear from the construction.
\end{proof}

\vspace{4mm}
\begin{para}\label{Pro:WeightFiltration}
Let $M=[Y\to A]$ be a 1--motive over $k$ where $A$ is an abelian variety. There are obvious morphisms of 1--motives 
$$A[0] \to M \to Y[1]$$ 
where $A[0]$ denotes the 1--motive $[0\to A]$ and $Y[1]$ denotes the 1--motive $[Y\to 0]$. These morphisms induce a short exact sequence of Galois representations as well as a short exact sequence of rational Hodge structures
$$0\to \Vell A \to \Vell M \to Y\otimes \IQ_\ell \to 0 \qqet 0\to \Vnot A \to \Vnot M \to Y\otimes\IQ \to 0$$
These exact sequences are compatible in the sense that the underlying exact sequence of $\IQ_\ell$--vector spaces of the $\ell$--adic realisations is canonically isomorphic to the underlying exact sequence of $\IQ$--vector spaces of the Hodge realisation tensored with $\IQ_\ell$. This follows from Proposition \ref{Pro:TateModIsRationalM}. Observe that $\Vell A$ is the usual $\ell$--adic Galois representation associated with $A$, obtained by tensoring the $\ell$--adic Tate module $\lim A(\overline k)[\ell^i]$ with $\IQ_\ell$, and that $\Vnot A$ is canonically isomorphic to the singular homology group $H_1(A(\IC),\IQ)$, which also is a rational Hodge structure of pure weight 1.
\end{para}

\vspace{4mm}
\begin{para}\label{Par:LieAlgExtLadic}
Let $M=[Y\to A]$ be a 1--motive over $k$ where $A$ is an abelian variety. Let $\overline k$ be an algebraic closure of $k$ and set $\Gamma:=\Gal(\overline k|k)$. We write $L^M$ and $L^A$ for the image of $\Gamma$ in the group of $\IQ_\ell$--linear automorphisms of $\Vell M$ and $\Vell A$ respectively, and we denote by $L^M_A$ the stabiliser of $\Vell A$ in $L^M$. We have thus a short exact sequence of $\ell$--adic Lie groups $0\to L^M_A \to L^M \to L^A \to 1$ and associated with it is a short exact sequence of Lie algebras
$$0\to\fl^M_A\to \fl^M \to \fl^A \to 0$$
The Lie algebra $\fl^M_A$ acts trivially on $Y\otimes \IQ_\ell$ and on $\Vell A$. Hence it is commutative and may be identified with a $\IQ_\ell$--linear subspace of $\Hom(Y\otimes\IQ_\ell, \Vell A)$. To determine $\fl^M$ amounts to determine the Lie algebras $\fl^A$ and $\fl^M_A$ and to determine how $\fl^M$ is an extension of $\fl^A$ by $\fl^M_A$. We can now formulate the main results of this section.
\end{para}

\vspace{4mm}
\begin{defn}\label{Def:fhMA}
For every a 1--motive $M=[u:Y\to A]$, where $A$ is an abelian variety, we write $\fh^M_A$ for the $\IQ$--linear subspace of $\Hom(Y\otimes\IQ, V_0A)$ consisting of those homomorphisms $f$ such that $\psi_1f(y_1) + \cdots + \psi_nf(y_n)=0$ whenever $\psi_i \in \End_{\overline k}A$ and $y_i\in Y$ are such that $\psi_1u(y_1) + \cdots + \psi_nu(y_n)=0$.
\end{defn}

\vspace{4mm}
\begin{thm}\label{Thm:RibetVersion}
Let $M=[u:Y\to A]$ be a 1--motive over $k$ where $A$ is an abelian variety. The equality $\fh^M_A \otimes \IQ_\ell = \fl^M_A$ holds for all prime numbers $\ell$. In particular the dimension of $\fl^M_A$ is independent of $\ell$.
\end{thm}

\vspace{4mm}
\begin{par}
The result is not really new, it essentially is a reformulation of a theorem of Ribet \cite{Ribe76} (see also \cite{Hind88}, Appendix 2). While the inclusion $\fh^M_A \otimes \IQ_\ell \supseteq \fl^M_A$ is elementary to show, the inclusion in the other direction uses Faltings's theorem on homomorphisms of abelian varieties over number fields (\cite{Faltings}) as well as Bogomolov's theorem on the image of the Galois group in the automorphisms of the Tate module of an abelian variety (\cite{Bogo81}).
\end{par}

\vspace{4mm}
\begin{para}\label{Par:TrailerfhM}
We will moreover construct a Lie subalgebra $\fh^M$ of $\End(\Vnot M)$ with the following properties. The Lie algebra $\fh^M$ leaves $\Vnot A$ invariant and acts trivially on $Y\otimes \IQ$. The stabiliser of $\Vnot A$ in $\fh^M$ is the Lie algebra $\fh^M_A$ defined in \ref{Def:fhMA}. So there is a short exact sequence
$$0\to\fh^M_A\to \fh^M \to \fh^A \to 1$$
where $\fh^A$ is the image of $\fh^M$ in the endomorphisms of $\Vnot A$. The Lie algebra $\fh^A$ is chosen in such a way that $\fh^M \otimes \IQ_\ell$ is contained in $\fl^M$, and in the case where  the equality $\fh^A\otimes \IQ_\ell = \fl^A$ holds, the equality $\fh^M\otimes \IQ_\ell = \fl^M$ holds as well. We would of course like to take for $\fh^A$ a Lie algebra such that for every prime number $\ell$ the equality 
$$\fh^A\otimes\IQ_\ell \overset{?}{=} \fl^A$$
holds. The Mumford--Tate conjecture states that such a Lie algebra exists and that it is the Lie algebra associated with the Mumford--Tate group of $A$. We do not want to assume this conjecture here.
\end{para}

\vspace{4mm}
\begin{para}{\bf Notation:} 
For a nontrivial abelian variety $A$ over $\overline k$ and every prime number $\ell$ we let $\fh^A = \fh^A_{(\ell)}$ denote any Lie subalgebra of $\End(\Vnot A)$ having the following properties
\begin{enumerate}
\item As an $\fh^A$--module $\Vnot A$ is semisimple.
\item The Lie algebra $\fh^A$ is contained in the commutator of $\End_{\overline k}(A)$ in $\End(\Vnot A)$.
\item The identity endomorphism of $\Vnot A$ belongs to $\fh^A$.
\item The Lie algebra $\fl^A$ \emph{contains} $\fh^A\otimes\IQ_\ell$.
\end{enumerate}
Such a Lie algebra indeed exists, we could just take $\fh^A$ to be the commutative 1--dimensional Lie algebra $\IQ$ acting as scalar multiplication on $\Vnot A$, independently of $\ell$. A Theorem of Bogomolov (\cite{Bogo81}, Theorem 3) asserts that the Lie algebra $\fl^A$ contains the scalars. Bogomolov's Theorem even assures that we can take $\fh^A$ such that the equality $\fl^A = \fh^A\otimes\IQ_\ell$ holds, but then $\fh^A$ might depend on $\ell$. If the Mumford--Tate conjecture holdsfor $A$ we can take $\fh^A$ to be the Lie algebra of the Mumford--Tate group of $A$.
\end{para}

\vspace{4mm}
\begin{para}\label{Par:RibetPreparation}
We now come to the proof of Theorem \ref{Thm:RibetVersion}, which we split up in several lemmas. We start with three preliminary remarks.

\begin{par}
{\bf (a)} In proving Theorem \ref{Thm:RibetVersion} we can without loss of generality replace $k$ by a finite extension of $k$. Indeed, if we do so the group $L^M$ gets replaced by a subgroup of finite index, which has then the same Lie algebra as $L^M$. In particular, we can and will assume from now on that $Y$ is constant and that all endomorphisms of $A$ are defined over $k$.
\end{par}

\begin{par}
{\bf (b)} The fppf--sheaf $\cHom(Y,A)$ on $\spec k$ is representable by a power of $A$. The morphism $u:Y\to A$ is a $k$--rational point on $\cHom(Y,A)$, and we denote by $B$ the connected component of the smallest algebraic subgroup of $\cHom(Y,A)$ containing $u$. In proving Theorem \ref{Thm:RibetVersion} we can without loss of generality suppose that $u$ belongs to $B$. Indeed, the smallest algebraic subgroup of $\cHom(Y,A)$ containing $u$ has only finitely many connected components because $\cHom(Y,A)$ is proper, hence for some $m>0$ the point $mu$ belongs to $B$. The morphism of 1--motives
$$[Y\xrightarrow{\:\:u\:\:} A] \xrightarrow{\:\:(m,\id)\:\:} [Y\xrightarrow{\:\:mu\:\:} A]$$
induces isomorphisms under the realisation functors $\Vell(-)$ and $\Vnot(-)$, so we may replace $u$ by $mu$.
\end{par}

\begin{par}
{\bf (c)} Let us write $E := \End_{\overline k}A\otimes\IQ$ and denote by $R$ the $\IQ$--linear subspace of $E\otimes Y$ generated by the elements $\psi_1\otimes y_1 +\cdots+\psi_n\otimes y_n \in \End_{\overline k}A\otimes Y$ such that $\psi_1u(y_1) + \cdots + \psi_nu(y_n)=0$ in $A(k)$. The subspace $R$ of $E\otimes Y$ is obviously an $E$--submodule. We have a canonical pairing
$$\angl{-,-}:(E\otimes Y)\times \Hom(Y\otimes\IQ,\Vnot A) \to \Vnot A$$
defined by $\angl{\psi\otimes y,f} = \psi f(y)$. By definition $\fh^M_A$ is the annihilator of $R$ in this pairing.
\end{par}
\end{para}

\vspace{4mm}
\begin{lem}
There is a canonical and natural isomorphism of $E$--modules $\Vnot\cHom(Y,A) \cong \Hom(Y\otimes\IQ, \Vnot A)$. Under this isomorphism $\Vnot B \subseteq \Vnot\cHom(Y,A)$ and $\fh^M_A \subseteq \Hom(Y\otimes\IQ, \Vnot A)$ correspond to each other.
\end{lem}

\begin{proof}
We choose a $\IZ$--basis $y_1, \ldots, y_r$ of $Y$ so that we can identify $Y$ with $\IZ^r$ and hence the abelian varieties $\cHom(Y,A)$ and $A^r$. This identification is natural in $A$, and the point $u$ of $\cHom(Y,A)$ corresponds to the point $(u(y_1), \ldots, u(y_r))$ of $A^r$. We get isomorphisms of $E$--modules
$$\Vnot\cHom(Y,A) \cong \Vnot(A^r) \cong  (\Vnot A)^r \cong \Hom(Y\otimes\IQ, \Vnot A) $$
whose composition is independent of the choice of the basis of $Y$. An element $x$ of $\Vnot(A^r) \subseteq \Lie A^r(\IC)$ belongs to $\Vnot B$ if and only if the one parameter subgroup $\exp(\IR x)$ of $A^r(\IC)$ is contained in $B(\IC)$. It follows from Poincar\'e's Reducibility Theorem (\cite{Mumf70} IV.19, Theorem 1) that a connected subgroup of $A^r(\IC)$ is contained in $B$ if and only if it is contained in $\ker\psi$ for every morphism $\psi:A^r\to A$ such that $\psi(B)=0$.  By minimality of $B$ we have $\psi(B)=0$ if and only if $\psi(u)=0$, hence we find
$$x\in \Vnot B \iff \psi(\exp(\IR x)) = 0 \:\pt  \psi\in\Hom(A^r,A) \mbox{ such that } \psi(u)=0$$
But now observe that $\psi(\exp(\IR x))= \exp(\IR \psi x)$ and that to say that $\exp(\IR \psi x) = 0$ is the same as to say that $\psi x=0$. If we denote by $\psi_1, \ldots, \psi_r$ the components of $\psi\in\Hom(A^r,A)$, we therefore  have 
$$x\in \Vnot B \iff \psi x = 0 \pt \psi_1, \ldots, \psi_r\in\End A \mbox{ with } \psi_1u(y_1) + \cdots + \psi_ru(y_r)=0$$
If we now look at $x \in \Vnot(A^r)$ as being a homomorphism $Y\otimes\IQ\to \Vnot A$ via the isomorphism we have introduced, the condition that $\psi x =0$ for all $\psi$ means that $x$ belongs to $\fh^M_A$. 
\end{proof}

\vspace{4mm}
\begin{lem}\label{Lem:RibetInclusion}
Let $M=[Y\to A]$ be a 1--motive over $k$ where $A$ is an abelian variety, and let $\ell$ be a prime number. The Lie algebra $\fl^M_A$ is contained in $\fh^M_A \otimes \IQ_\ell$.
\end{lem}

\begin{proof}
Let $r=\psi_1\otimes y_1 +\cdots+\psi_n\otimes y_n$ be an element of $R$ and let us show that we have $\angl{r,x}=0$ for every $x\in \fl^M_A$. Replacing $r$ by some multiple of $r$ we may suppose that the $\psi_i$ are actual endomorphisms of $A$. We must show that for every $\sigma\in L^M_A$ we have $\angl{r,\log \sigma}=0$. We have $\log\sigma= \sigma-1$, so what we have to show is that for all $\sigma\in\Gal(\overline k|k)$ acting trivially on $\Tell A$ we have $\angl{r,\sigma-1} = 0$. For every $y_i$, let $v_i$ be an element of $\Tell M$ mapping to $y_i \otimes 1$ in $Y\otimes\IZ_\ell$. Using our explicit description of the Tate module $\Tell M$ given in \ref{Par:TellMExplicit} we may write these preimages as sequences $v_i = (P_{ij},y_i)_{j=0}^{\infty}$ where the $P_{ij}\in A(\overline k)$ are points such that $P_{i0} = u(y_i)$ and $\ell P_{i, j+1} = P_{ij}$ for all $j\geq 0$. Now we compute
$$\angl{r,\sigma-1} = \sum_{i=1}^n\psi_i (\sigma v_i-v_i) = \sum_{i=1}^n\psi_i (\sigma P_{ij} - P_{ij})_{j=0}^{\infty} =  \sigma \sum_{i=1}^n(\psi_iP_{ij})_{j=0}^\infty - \sum(\psi_iP_{ij})_{j=0}^{\infty} $$
By definition of $R$ we have $\psi_1P_{10} + \cdots + \psi_n P_{n0}=0$ hence $\psi_1P_{1j} + \cdots + \psi_n P_{nj}$ is an element of order $\ell^j$ in $A(\overline k)$. But by hypothesis $\sigma$ acts trivially on $\Tell A$, hence on all $\ell^j$--torsion points of $A(\overline k)$. Therefore, the right hand side of the above equality is zero.
\end{proof}

\vspace{4mm}
\begin{lem}\label{Lem:BogoHSSS}
Let $M=[Y\to A]$ be a 1--motive over $k$ where $A$ is an abelian variety, and let $\ell$ be a prime number. There is a canonical isomorphism $H^1(L^M, \Vell A) \cong \Hom_{L^A}(L^M_A, \Vell A)$. 
\end{lem}

\begin{proof}
The Hochschild--Serre spectral sequence furnishes an exact sequence in low degrees
$$0\to H^1(L^A,\Vell A) \to H^1(L^M,\Vell A) \xrightarrow{\:\:(\ast)\:\:} H^0(L^A, H^1(L^M_A, \Vell A)) \to H^2(L^A,\Vell A)$$
By Bogomolov's theorem (\cite{Bogo81} Theorem 3) there exists an element in $L^A$ which acts as multiplication by a scalar $\neq 1$ on $\Vell A$. Thus, by Sah's Lemma the first and last term in the above exact sequence vanish, and so the map labelled $(\ast)$ is an isomorphism. Since $L^M_A$ acts trivially on $\Vell A$ by definition, we have $H^0(L^A, H^1(L^M_A, \Vell A)) = \Hom_{L^A}(L^M_A, \Vell A)$. 
\end{proof}

\vspace{4mm}
\begin{lem}\label{Lem:DividingInjection}
There is a canonical, injective $\IZ_\ell$--linear map 
$$\Hom_k(B,A)\otimes\IZ_\ell\to H^1(L^M,\Tell M)$$ 
\end{lem}

\begin{proof}
Let us write $k_M$ for the field extension of $k$ whose Galois group is the quotient $L^M$ of $\Gamma = \Gal(\overline k|k)$. By our explicit description of the Tate module of $M$ (\ref{Par:TellMExplicit}), this $k_M$ is the smallest field extension of $k$ such that for all $y\in Y$ all $\ell$--division points of $u(y)$ are defined over $k_M$. In other words, $k_M$ is the smallest extension of $k$ such that all elements of $u(Y)$ become $\ell$--divisible in $A(k_M)$. Any point $P\in A(k)$ which is an $\End_kA$--linear combination of points in $u(Y)$ becomes then divisible in $A(k_M)$ as well. We consider now the following diagram
$$\begin{diagram}
\setlength{\dgARROWLENGTH}{4mm}
\node[3]{\Hom_k(B,A)\otimes \IZ_\ell}\arrow{s,r}{(1)}\arrow{sw,--}\\
\node{0}\arrow{e}\node{K}\arrow{s,r}{(4)} \arrow{e}\node{A(k)\otimes \IZ_\ell} \arrow{s,r}{(5)}\arrow{e,t}{(2)}\node{A(k_M)\hotimes \IZ_\ell}\arrow{s,r}{(6)}\\
\node{0}\arrow{e}\node{H^1(L^M, \Tell A)}\arrow{e}\node{H^1(k,\Tell A)} \arrow{e,t}{(3)}\node{H^1(k_M,\Tell A)}
\end{diagram}$$
Let me explain the maps. First, the map (1) is induced by the map $\Hom_k(B,A) \to A(k)$ sending a homomorphism $\varphi$ to the rational point $\varphi(u)$. The maps (2) and (3) are induced by the inclusion of fields $k \subseteq k_M$. We use here that $A(k)$ is finitely generated, so $A(k)\otimes \IZ_\ell$ is the same as $A(k)\hotimes \IZ_\ell$. The vertical maps (5) and (6) are the maps in the Kummer sequences introduced in \ref{Par:IntroTateMod} (for $i=1$), so they are both injective. We define $K$ to be the kernel of (2). From the Hochschild--Serre spectral sequence we see that the kernel of (3) is $H^1(L^M, \Tell A)$. The map (4) is then the restriction of (5) so that the diagram commutes. Since (5) is injective, (4) is injective as well.\\
Having this diagram, all that remains to show is that the dashed arrow exists and that it is injective. In other words, we have to show that (1) is injective and that the composition of (1) and (2) is zero. The map (1) is injective because $\IZ_\ell$ is a flat $\IZ$--module and because already the map $\Hom_k(B,A) \to A(k)$ is injective. Indeed, let $\varphi:B\to A$ be a morphism of abelian varieties such that $\varphi(u)=0 \in A(k)$. The kernel of $\varphi$ is then an algebraic subgroup of $B$ containing $u$, hence equal to $B$ by minimality of $B$, and so $\varphi$ is zero. The composition of (1) and (2) is zero. Indeed, for every homomorphism $\varphi:B\to A$ the point $\varphi(u)$ is an $\End_kA$--linear combination of points in $u(Y)$, hence $\varphi(u)$ is $\ell$--divisible in $A(k_M)$, and hence the class of $\varphi(u)$ in $A(k_M)\hotimes \IZ_\ell$ is trivial.
\end{proof}

\vspace{4mm}
\begin{rem}
Explicitly, the map whose existence we claim in the lemma is the following. Given a homomorphism $\varphi:B\to A$, it sends $\varphi\otimes 1$ to the class of the cocycle $$c_\varphi : \sigma\mapsto (\sigma P_i-P_i)_{i=0}^\infty \in\Tell A$$
where $(P_i)_{i=0}^\infty$ is a sequence of points in $A(\overline k)$ such that $P_0=\varphi(u)$ and $\ell P_{i+1}=P_i$. As we shall see in a moment, this map has a finite cokernel. It is then not hard to see that the points of $P\in A(k)$ which become divisible in $A(k_M)$ are precisely those points such that that for some integer $m>0$ the point $mP$ is an $\End_kA$--linear combination of points in $u(Y)$. This relates Theorem \ref{Thm:RibetVersion} with Ribet's Main Theorem in \cite{Ribe76} on dividing points on abelian varieties.
\end{rem}

\vspace{4mm}
\begin{proof}[Proof of Theorem \ref{Thm:RibetVersion}]
By Faltings's theorem on homomorphisms of abelian varieties over number fields, and because we suppose that all endomorphisms of $A$ are defined over $k$, we have a canonical isomorphism $\Hom_k(B, A) \otimes \IQ_\ell \cong \Hom_{\fl^A}(\Vell B, \Vell A)$ By Lemma \ref{Lem:BogoHSSS} we have a canonical isomorphism $H^1(L^M, \Vell A) \cong \Hom_{L^A}(L^M_A, \Vell A)$. Together with Lemma \ref{Lem:DividingInjection} this yields an injection
$$\Hom_{\fl^A}(\Vell B, \Vell A) \cong \Hom_k(B,A)\otimes\IQ_\ell \to \Hom_{\fl^A}(\fl^M_A, \Vell A)$$
We have seen in Lemma \ref{Lem:RibetInclusion} that the inclusion $\fl^M_A\subseteq \fh^M_A \otimes \IQ_\ell \cong \Vnot B\otimes \IQ \cong \Vell B$ holds. Let us then consider the restriction map
$$\Hom_{\fl^A}(\Vell B, \Vell A) \to \Hom_{\fl^A}(\fl^M_A, \Vell A)$$
Because $\Vell A$, $\Vell B$ and $\fl^M_A$ are all semisimple $\fl^A$--modules by Faltings's results, this map is surjective and it is injective if and only if the equality $\fl^M_A = \Vell B$ holds. This is indeed the case, for dimension reasons.
\end{proof}

\vspace{4mm}
\begin{para}\label{Par:ConstrSection}
We now come to the construction of the Lie algebra $\fh^M \subseteq \End(\Vnot M)$ which will be an extension of $\fh^A$ by $\fh^M_A$ as announced in \ref{Par:TrailerfhM}. Let $M=[u:Y\to A]$ be a 1--motive over $k$ where $A$ is an abelian variety, and consider the 1--motive
$$M_+=[u_+:\End_{\overline k}A\otimes Y\to A]$$
given by $u_+(\psi\otimes y) = \psi u(y)$. There is a canonical morphism of 1--motives $M\to M_+$ inducing a diagram
$$\begin{diagram}
\setlength{\dgARROWLENGTH}{4mm}
\node{0}\arrow{e}\node{\Vnot A}\arrow{s,=}\arrow{e,t}{\subseteq}\node{\Vnot M}\arrow{s,J}\arrow{e,t}{p}\node{Y\otimes \IQ}\arrow{s,J}\arrow{e}\node{0}\\
\node{0}\arrow{e}\node{\Vnot A}\arrow{e}\node{\Vnot M_+}\arrow{e,t}{p_+}\node{\End_{\overline k}A\otimes Y\otimes \IQ}\arrow{e}\node{0}\\
\node[4]{\ker u_+\otimes\IQ}\arrow{nw,b}{\natural}\arrow{n,r}{\subseteq}
\end{diagram}
$$
Because the map $u_+$ is a map of $\End_{\overline k}A$--modules, the maps in the lower exact sequence as well as the canonical lift $\natural$ (cf. \ref{Par:ConstructionVnotM}) are maps of $E := \End_{\overline k}A\otimes\IQ$--modules. Because $E$ is a semisimple $\IQ$--algebra (\cite{Mumf70}, IV.19 Theorem 1) we can choose an $E$--module section $s_+$ of $p_+$ extending $\natural$. Denote by $s$ the restriction of $s_+$ to $Y\otimes\IQ$. This $s$ takes values in $\Vnot M$ and is therefore a section of $p$. We now give the definition of $\fh^M$ and proceed then with checking that this definition makes sense.
\end{para}

\vspace{4mm}
\begin{defn}\label{Def:fhM}
Let $s$ be a section of the canonical projection $\Vnot M\to Y\otimes\IQ$ such as constructed in \ref{Par:ConstrSection}. We define $\fh^M$ to be the Lie subalgebra of $\End(\Vnot M)$ consisting of those endomorphisms which are of the form
$$(e,f)_s: v + s(y) \mapsto ev + f(y) \qquad \pt v\in \Vnot A \subseteq \Vnot M, \: y\in Y\otimes\IQ$$
for some $e \in \fh^A$ and some $f\in\fh^M_A \subseteq \Hom(Y\otimes\IQ,\Vnot A)$.
\end{defn}

\vspace{4mm}
\begin{prop}
The set of endomorphisms $\fh^M$ of $\Vnot M$ defined in \ref{Def:fhM} is indeed a Lie subalgebra of $\End(\Vnot M)$. Moreover, $\fh^M$ does not depend on the choice of the section $s$.
\end{prop}

\begin{proof}
The set $\fh^M$ is a linear subspace of $\End(\Vnot M)$. In order to show that $\fh^M$ is a Lie subalgebra we must show that $\fh^M$ is closed under taking commutators. Indeed, the formula $[(e,f)_s,(e',f')_s] = ([e,e'], e\circ f'-e'\circ f)_s$ holds, and $e\circ f'-e'\circ f$ is again an element of $\fh^M_A$ because the composition of $f\in \fh^M_A$ with any endomorphism of $\Vnot A$ again belongs to $\fh^M_A$ by definition of $\fh^M_A$. We now show that $\fh^M$ is independent of $s$. Consider again the diagram of \ref{Par:ConstrSection}, let $s_+$ and $t_+$ be $E$--module sections of $p_+$ extending $\natural$ and write $s$ and $t$ for their restrictions to $Y\otimes\IQ$. We claim that the difference $d := s-t:Y\otimes\IQ\to\Vnot A$ belongs to $\fh^M_A$. Indeed, observe that the objects introduced in \ref{Par:RibetPreparation}.c reappear in the diagram of \ref{Par:ConstrSection}, namely
$$\End_{\overline k}A\otimes Y\otimes \IQ = E \otimes Y \qqet \ker u_+\otimes\IQ = R$$
We have $\angl{d,r}=0$ for all $r\in R$ because $s_+$ and $t_+$ are $E$--module maps that coincide on $R$, and that means by definition that $d$ belongs to $\fh^M_A$. From this we can deduce that the Lie algebras constructed as in the definition \ref{Def:fhM} from $s$ and from $t$ respectively are the same. Indeed, the equalities
$$(e,f)_s = (e,f-e\circ d)_t \qqet (e,f)_t = (e,f+e\circ d)_s$$
hold for all $e \in \fh^A$ and all $f\in\fh^M_A \subseteq \Hom(Y\otimes\IQ,\Vnot A)$. We have seen that $d$ belongs to $\fh^M_A$ hence so do $f-e\circ d$ and $f+e\circ d$. That does it.
\end{proof}

\vspace{4mm}
\begin{cor}[To Theorem \ref{Thm:RibetVersion}]\label{Cor:RibetVersionflM}
Let $M=[u:Y\to A]$ be a 1--motive over $k$ where $A$ is an abelian variety and let $\ell$ be a prime number. The Lie algebra $\fl^M$ contains $\fh^M\otimes \IQ_\ell$, and the equality $\fl^M = \fh^M\otimes \IQ$ holds if and only if the equality $\fl^A = \fh^A\otimes \IQ_\ell$ holds.
\end{cor}

\begin{proof}
Define $M_+$ and choose $s_+$ as in \ref{Par:ConstrSection}, and construct the Lie algebra $\fh^M$ as in Definition \ref{Def:fhM} from this data. We still denote by $s_+$ and by $s$ the $\IQ_\ell$--linear extensions of $s_+$ and $s$, so we have a split short exact sequence of $\IQ_\ell$--vector spaces
\begin{center}
$ $ \xymatrixrowsep{1.5cm} \xymatrixcolsep{1.5cm} \xymatrix{
0 \ar@{->}[r] & {\Vell A}\ar@{->}[r]^{\subseteq}  & {\Vell M}\ar@{->}[r]_{p} & {Y\otimes \IQ_\ell} \ar@/u1.5ex/_{s}[l] \ar@{->}[r]& 0}  $ $
\end{center}
The $\fl^A$--module $\fl^M_A$ can be identified with a submodule of $\Hom(Y\otimes \IQ_\ell,\Vell A) \simeq \Vell A^r$. Since $\Vell A$ is a semisimple $\fl^A$--module by Faltings's results,  $\fl^M_A$ is isomorphic as an $\fl^A$--module to a direct factor of a power of $\Vell A$. Bogomolov's Theorem (\cite{Bogo81}, Theorem 3) and Sah's Lemma imply that 
$$H^i(\fl^A, \Vell A) = 0 \qqet H^i(\fl^A, \Hom(Y\otimes \IQ_\ell,\Vell A)) = 0 \qqet H^i(\fl^A, \fl^M_A) = 0$$
for all $i\geq 0$. The vanishing of $H^2(\fl^A, \fl^M_A)$ implies that the Lie algebra extension given in \ref{Par:LieAlgExtLadic} is split (\cite{Weib94}, theorem 7.6.3), we can therefore choose a splitting $\sigma$ of the projection map $\pi$ as indicated.
\begin{center}
$ $ \xymatrixrowsep{1.5cm} \xymatrixcolsep{1.5cm} \xymatrix{
0 \ar@{->}[r] & {\fl^M_A} \ar@{->}[r]^{\subseteq}  & {\fl^M}\ar@{->}[r]_{\pi} & {\fl^A} \ar@/u1.5ex/_{\sigma}[l] \ar@{->}[r]& 0}  $ $
\end{center}
Using the splittings $s$ and $\sigma$ we fabricate a map $c:\fl^A \to \Hom(Y\otimes \IQ,\Vell A)$ by setting
$$c(x)(v) = \sigma(x)s(v) \qquad \pt x\in\fl^A, v\in Y\otimes \IQ_\ell$$
This map is a cocycle, hence a coboundary because $H^1(\fl^A, \Hom(Y\otimes \IQ_\ell,\Vell A))$ vanishes. So, there exists a $\IQ_\ell$--linear map $f:Y\otimes \IQ_\ell \to \Vell A$ such that 
$$\sigma(e)s(v) = e.f(y) \qquad \pt e\in\fl^A, y\in Y\otimes \IQ_\ell$$
We claim this $f$ belongs to $\fl^M_A$. in order to check this it suffices by Theorem  \ref{Thm:RibetVersion} to show that for all $y_1, \cdots y_n\in Y$ and all $\psi_1, \ldots, \psi_n \in\End_{\overline k}A$ such that $\psi_1u(y_1)+\cdots + \psi_nu(y_n) = 0$ we have $\psi_1f(y_1)+\cdots + \psi_nf(y_n) = 0$. Indeed, we have
$$\sum_{i=1}^n\psi_if(y_i)= \sum_{i=1}^n\psi_i \sigma(x)s(y_i) = \sigma(x).s_+\Big(\sum_{i=1}^n \psi_i \otimes y_i\Big)$$
Here we have used that the $\psi_i$ commute with elements of $\fl^M$ and $\End_{\overline k}A$--linearity of $s_+$. By hypothesis $s_+$ sends elements of $\ker u_+ \otimes \IQ_\ell$ to $(\Vell M)^{\fl^M}$,  hence the right hand side of the above equation is zero. The map $\fl^A \to \fl^M$ given by $x \mapsto \sigma(x)-x.f$ is therefore another section of $\pi$. Let us replace $\sigma$ by this new section. By construction we have now $\sigma(e)s(y) = 0$ for all $e\in\fl^A$ and all $y\in Y\otimes \IQ_\ell$, hence
$$(\sigma(e) + f).(v + s(y)) = ev + f(y) \qquad\pt e\in\fl^A, f\in \fl^M_A, v\in\Vell A, y\in Y\otimes\IQ_\ell$$
Since $\fl^A$ contains $\fh^A\otimes \IQ_\ell$ and $\fl^M_A$ is equal to $\fh^M_A\otimes \IQ_\ell$, this shows that $\fl^M$ contains $\fh^M\otimes \IQ_\ell$, and that the equality $\fl^M = \fh^M\otimes \IQ$ holds if and only if the equality $\fl^A = \fh^A\otimes \IQ_\ell$ holds.
\end{proof}

\vspace{4mm}
\begin{rem}
We have left two important things undiscussed. First, we have only worked with 1--motives whose semiabelian part is an abelian variety. The benefit we had from this was Poincar\'es Reducibility Theorem and semisimplicity of various objects associated with the abelian variety. It would of course be desirable to have a statement as Corollary \ref{Cor:RibetVersionflM} for general 1--motives. Secondly, we have given the Lie algebra $\fh^M$ by an ad hoc construction. This construction should be compared with the Mumford--Tate group associated with the mixed Hodge structure $\Vnot M$, which one may define directly in terms of Tannakian formalism.
\end{rem}

%===============================================================================
\vspace{14mm}%==================================================================
\section{Some linear algebra}%==================================================
%===============================================================================

\begin{par}
The 1--motives we are working with in this section are of the form $M=[Y\to A]$ where $A$ is a geometrically simple abelian variety over $k$. I recall that this means that $A$ has no abelian subvariety defined over $\overline k$ other than $0$ and itself. Our goal is to prove the following technical result.
\end{par}

\begin{prop}\label{Pro:LinearImageEll}
Let $M=[Y\to A]$ be a 1--motive over $k$ where $A$ is a geometrically simple abelian variety, and let $\ell$ be a prime number. The image of the bilinear map
$$\alpha_\ell: (\Vell M)^\ast \times \fl^M \to (\Vell M)^\ast $$
given by $\alpha_\ell(\pi,x) = \pi\circ x$ consists precisely of those linear forms on $\Vell M$ which are zero on the subspace $\ker u\otimes\IQ_\ell$ of $\Vell M$. In particular, the image of $\alpha_\ell$ is a linear subspace of $(\Vell M)^\ast$. 
\end{prop}

\vspace{4mm}
\begin{para}\label{Par:LinalgSetup}
Here is the setup for this section. We fix a finite dimensional division algebra $E$ over $\IQ$, a nontrivial $E$--module $V_1$ of finite rank and a $\IQ$--vector space of finite dimension $V_0$. There is a canonical pairing
$$\angl{-,-}:(E\otimes V_0)\times \Hom(V_0,V_1) \to V_1$$
given by $\angl{\psi\otimes y,f} = \psi f(y)$. Furthermore, we fix an $E$--submodule $R$ of $E\otimes V_0$ and define $\fh_R \subseteq \Hom(V_0,V_1)$ to be the annihilator of $R$ in this pairing. The following proposition remains valid if one replaces $E$ by a finite product of division algebras over $\IQ$ -- the price to pay are more indices.
\end{para}

\vspace{4mm}
\begin{prop}\label{Pro:LinearImageAbstract}
In the situation of \ref{Par:LinalgSetup}, let $\pi$ be a nonzero linear form on $V_1$ and let $v$ be an element of $V_0$. The equality $\pi(f(v)) = 0$ holds for all $f\in \fh_R$ if and only if $1_E\otimes v$ belongs to $R$.
\end{prop}

\begin{proof}
\begin{par}
If $1_E\otimes v$ belongs to $R$ then $f(v)=0$ for all $f\in\fh_R$ by definition, so the \emph{if} part is obvious. To prove the converse, let us fix an element $v\in V_0$ such that 
$$\pi f(v)=0 \qquad\pt f\in\fh_R$$
We must show that $1_E\otimes v$ belongs to $R$. Let us choose a $\IQ$--basis of $V_0$ as follows. We begin by choosing elements $y_1, \ldots, y_r\in V_0$ such that $1_E\otimes y_1, \ldots, 1_E\otimes y_r$ form an $E$--basis of $(E\otimes V_0)/R$. These elements are $K$--linearly independent, hence we can choose elements  $z_1, \ldots, z_s$ of $V_0$ completing $y_1, \ldots, y_r$ to  basis of $V_0$. There exist unique elements $\psi_{ij}$ of $E$ such that for all $1\leq j \leq s$
$$r_j := 1_E \otimes z_j - (\psi_{j1} \otimes y_1 +\cdots + \psi_{jr}\otimes y_r)$$
belongs to $R$. We claim that a homomorphism $f:V_0\to V_1$ belongs to $\fh_R$ if and only if the relations
$$f(z_j)= \psi_{j1}f(y_1) +\cdots + \psi_{jr}f(y_r) \qquad\qquad \pt 1\leq j \leq s$$
hold. In other words we claim that $f$ belongs to $\fh_R$ if and only if $\angl{r_i,f}=0$ holds for $1\leq j\leq s$. Indeed, since $r_j\in R$, every $f\in \fh_R$ must satisfy $\angl{f,r_j}=0$ by definition. On the other hand, we must show that if $\angl{r_j,f} = 0$ holds for $1\leq j\leq s$, then we have $\angl{r,f} =0$ for all $r\in R$. This is the case because $R$ is $E$--linearly generated by $r_1, \ldots, r_s$. Indeed, we can write every $r\in R$ as $r=\psi_{1j}\otimes y_1 +\cdots + \psi_{rj}\otimes y_r + \varphi_1\otimes z_1 + \cdots + \varphi_s\otimes z_s$. After subtracting $\varphi_1 r_1 + \cdots + \varphi_sr_s$ from $r$ we remain with an element $r'\in R$ of the form $r'=\psi_{1j}'\otimes y_1 +\cdots + \psi_{rj}'\otimes y_r$. But this element can only be zero because the $1_E\otimes y_1, \ldots, 1_E\otimes y_r$ are an $E$--basis of $(E\otimes V_0)/R$.
\end{par}

\begin{par}
In summary, if we want to give an element $f\in\fh \subseteq \Hom(V_0,V_1)$, we may freely choose the values $f(y_1), \ldots, f(y_r) \in V_1$, and must then follow the rules $f(z_j)= \psi_{1j}f(y_1) +\cdots + \psi_{rj}f(y_r)$ to determine the value of $f$ on the remaining basis elements $z_1, \ldots, z_s$.
\end{par}

\begin{par}
Let us write $v=\alpha_1y_1 + \cdots + \alpha_ry_r + \beta_1z_1 + \cdots + \beta_sz_s$ for scalars $\alpha_i$ and $\beta_j\in \IQ$, and define elements $\rho_1, \ldots, \rho_r$ of $E$ by
$$\rho_i := \alpha_i1_E+\beta_1\psi_{1i} + \cdots + \beta_s\psi_{si}$$
for $1\leq i \leq r$. Using these definitions, the relation $\pi(f(v))=0$ becomes
$$0 = \pi\bigg(\sum_{i=1}^r \alpha_i f(y_i) + \sum_{j=1}^s \beta_jf(z_j)\bigg) = \pi\bigg(\sum_{i=1}^r \alpha_i f(y_i) + \sum_{i=1}^r\sum_{j=1}^s \beta_j\psi_{ji}f(y_i)\bigg) = \pi\sum_{i=1}^r \rho_i f(y_i)$$
For every $1\leq i \leq r$ and every $x\in V_1$ there exists an $f\in \fh_R$ such that $f(y_i)=x$ and $f(y_k)=0$ for $k\neq i$. The above relation shows thus in particular that $\pi(\rho_i(x))=0$ for all $x\in V_1$, that is, $\pi\circ \rho_i =0$. Since $\pi$ is nonzero, this means that $\rho_i$ is not invertible, and since $E$ is a division algebra, we find $\rho_i=0$. Thus, the equality 
$$0=\alpha_i1_E \otimes y_i +\beta_1\psi_{1i} \otimes y_i + \cdots + \beta_s\psi_{si} \otimes y_i$$
holds in $E\otimes V_0$ for all $1\leq i \leq r$. Summing over all $i$ yields then
$$0=\sum_{i=1}^r\alpha_i1_E \otimes y_i +\sum_{j=1}^s\beta_j\sum_{i=1}^r\psi_{1i} \otimes y_i = \underbrace{\sum_{i=1}^r\alpha_i1_E \otimes y_i +\sum_{j=1}^s\beta_j1_E \otimes z_j}_{1_E\otimes v} - \sum_{j=1}^s\beta_jr_j$$
Hence $1_E\otimes v = \beta_1r_1 + \cdots + \beta_sr_s$ belongs to $R$, and that is what we wanted to show.
\end{par}
\end{proof}

\vspace{4mm}
\begin{prop}\label{Pro:LinearImageConcrete}
Let $M=[u:Y\to A]$ be a 1--motive over $\overline k$ where $A$ is a simple abelian variety. The image of the bilinear map
$$\alpha_0: (\Vnot M)^\ast \times \fh^M \to (\Vnot M)^\ast $$
given by $\alpha_0(\pi,x) = \pi\circ x$ consists precisely of those linear forms on $\Vnot M$ which are zero on the subspace $\ker u\otimes\IQ$ of $\Vnot M$. In particular, the image of $\alpha_0$ is a linear subspace of $(\Vnot M)^\ast$. 
\end{prop}

\begin{proof}
Let us fix a linear section $s:(Y\otimes \IQ)\to\Vnot M$ such as in the construction of $\fh^M$, so that every element of $\fh^M$ is of the form
$$(e,f)_s: v + s(y) \mapsto ev + f(y) \qquad \pt v\in \Vnot A, \: y\in Y\otimes\IQ$$
for some $e \in \fh^A$ and some $f\in\fh^M_A$. Using this section, every linear form $\pi$ on $\Vnot M$ can be uniquely written as  $\pi = (\pi_A,\pi_Y)$, where $\pi_A$ is a form on $\Vnot A$ and $\pi_Y$ is a form on $Y\otimes \IQ$. With this notation, the map $\alpha_0$ in the proposition becomes
$$\alpha_0: \big((\pi_A,\pi_Y), (e,f)_s\big) \mapsto (\pi_A\circ e, \pi_A\circ f)$$
For every linear form $(\pi_A,\pi_Y)$ on $\Vnot M$, every element $(e,f)_s$ of $\fh^M$ and every $y\in\ker u\otimes \IQ$ we have $(\pi_A\circ e, \pi_A\circ f)_s(0,s(y)) = f(y)=0$ by definition of $\fh^M_A$, so all forms in the image of $\alpha_0$ annihilate $\ker u\otimes \IQ$. On the other hand, let $(\eta_A,\eta_Y)$ be a linear form on $\Vnot M$ such that $\eta_Y(y)=0$ for all $y\in\ker u$. Let us define
$$e := \begin{cases}\id & \mbox{if $\eta_A\neq 0$}\\ 0 & \mbox{if $\eta_A = 0$} \end{cases} \qqet (\pi_A, \pi_Y) := \begin{cases}(\eta_A,0) & \mbox{if $\eta_A\neq 0$}\\ (\pi_A,0) \mbox{ for some $\pi_A\neq 0$} & \mbox{if $\eta_A = 0$} \end{cases}$$
In order to make use of Proposition \ref{Pro:LinearImageAbstract}, we specialise the objects introduced in \ref{Par:LinalgSetup} as follows. We take $E$ to be the $\IQ$--algebra $\End_{\overline k}(A)\otimes \IQ$, which is a division algebra according to \cite{Mumf70}, IV.19 Corollary 2 to Theorem 1. Then $V_1 := \Vnot A$ is an $E$--module of finite rank, and we specialise $V_0 := Y\otimes\IQ$. Finally we let $R$ be the $E$--submodule of $E\otimes (Y\otimes\IQ)$ introduced in \ref{Par:RibetPreparation}.c, so that according to Definition \ref{Def:fhMA} we have $\fh_R=\fh^M_A$. Proposition \ref{Pro:LinearImageAbstract} states that the image of the \emph{linear} map $\fh^M_A \to (Y\otimes \IQ)^\ast$ given by $f\mapsto \pi_A\circ f$ is equal to the annihilator of the subspace $\ker u\otimes \IQ$ of $Y\otimes\IQ$. In particular there exists an element $f\in \fh^M_A$ such that $\pi_A\circ f = \eta_Y$. With this choice of $f$ we have
$$\alpha_0\big((\pi_A,\pi_Y), (e,f)_s\big) = (\pi_A\circ e, \pi_A\circ f) = (\eta_A,\eta_Y)$$
in both cases, $\eta_A=0$ and $\eta_A\neq 0$. This proves the proposition. 
\end{proof}

\vspace{4mm}
\begin{para}
It follows from Theorem \ref{Thm:RibetVersion} (or rather its Corollary \ref{Cor:RibetVersionflM}) that the $\IQ_\ell$--bilinear map in Proposition \ref{Pro:LinearImageEll} is obtained from the $\IQ$--bilinear map of Proposition \ref{Pro:LinearImageConcrete} by extension of scalars. However, it is not clear whether or not the property of a bilinear map to be surjective is invariant under scalar extension. Let $L|K$ be an extension of fields. Given finite dimensional $K$--vector spaces $U,V,W$ and a $K$--bilinear map $\beta_K:U\times V\to W$, denote by $\beta_L$ the $L$--bilinear map obtained from $\beta_K$. Which of the following implications is true (for a fixed field extension $L|K$ and all $K$--bilinear maps $\beta_K$ between finite dimensional $K$--vector spaces)?
$$\mbox{$\beta_K$ is surjective} \quad \overset{a)}{\Longleftarrow} \:\: \overset{b)}{\Longrightarrow} \quad \mbox{$\beta_L$ is surjective}$$
We were unable to find a satisfying answer to this general problem. Our next proposition shows that the implication b) holds for the extension $\IQ_\ell|\IQ$, and that is all we need.
\end{para}

\vspace{4mm}
\begin{prop}\label{Pro:BilinearScalarExt}
Let $V,V'$ and $W$ be $\IQ$--vector spaces and let $\alpha:V\times V'\to W$ be a bilinear map. Let $K$ be either the field of real numbers or the field of $\ell$--adic numbers for some prime number $\ell$. If the image of $\alpha$ is a linear subspace of $W$, then the image of the induced $K$--bilinear map
$$\alpha_K: (V\otimes K) \times (V'\otimes K) \to W \otimes K$$
is a linear subspace of $W\otimes K$, and the equality $\im\alpha_K=\im\alpha\otimes K$ holds.
\end{prop}

\begin{proof}
To ease notation let us define $V_K := V\otimes K$ and analogously $V'_K$ and $W_K$. The image of $\alpha_K$ is certainly contained in the $K$--linear subspace $\im\alpha\otimes K$. We may thus, replacing $W$ by $\im\alpha$, suppose without loss of generality that $\alpha$ is surjective. We have to show that $\alpha_K$ is surjective as well. We consider the projective spaces
$$\IP V := (V\setminus \{0\})/\IQ^\ast \qqet \IP V_K := (V_K\setminus \{0\})/K^\ast$$
Because $\IQ$ is dense in $K$, the subset $\IP V$ is dense in $\IP V_K$, and again the same goes for $V'$ and $W$ in place of $V$. The map $\alpha$ induces well defined maps 
$$\overline\alpha : \IP V \times \IP V' \to \IP W \qqet \overline\alpha_K : \IP V_K\times \IP V'_K \to \IP W_K$$
Considering $\IP V \times \IP V'$ as a subset of $\IP V_K\times \IP V'_K$, the map $\overline\alpha$ extends to $\overline\alpha_K$, hence in particular the image of $\overline \alpha$ contains the dense subset $\IP W$ of $\IP W_K$. On the other hand, the topological spaces $\IP V_K$ and $\IP V_K'$ are compact, hence so is their product, and the map $\overline \alpha_K$ is continuous. Thus, the image of $\overline \alpha_K$ must be compact, hence closed, and therefore consist of all of $\IP W_K$. But then, surjectivity of $\alpha_K$ immediately follows from surjectivity of $\overline \alpha_K$.
\end{proof}

\vspace{4mm}
\begin{proof}[Proof of Proposition \ref{Pro:LinearImageEll}]
On one hand, let $\pi$ be a linear form on $\Vell M$ and let $x$ be an element of $\fl^M$. For every $v\in\ker u\otimes\IQ_\ell \subseteq \Vell M$ we have $x.v=0$ and hence $\pi(x.v)=0$. On the other hand, let $\eta$ be a linear form on $\Vell M$ which is trivial on $\ker u \otimes \IQ_\ell$. By Corollary \ref{Cor:RibetVersionflM} the Lie algebra $\fl^M$ contains $\fh^M \otimes \IQ_\ell$, hence it is enough to show that the image of the bilinear map
$$(\Vell M)^\ast \times (\fh^M\otimes \IQ_\ell) \to (\Vell M)^\ast $$
contains all linear forms on $\Vell M \cong \Vnot M \otimes \IQ_\ell$ which are trivial on $\ker u\otimes\IQ_\ell$. Indeed, that follows from Proposition \ref{Pro:LinearImageConcrete} and Proposition \ref{Pro:BilinearScalarExt}.
\end{proof}

%===============================================================================
\vspace{14mm}%==================================================================
\section{Proof of the Main Theorem}%============================================
%===============================================================================

\begin{par}
For this section we prove our main theorem as announced in the introduction. Our strategy is as follows: Given a geometrically simple abelian variety $A$ over the number field $k$ and a subgroup $X$ of $k$, we consider the group 
$$\overline X := \{P\in A(k) \tq \red_\fp(P) \in \red_\fp(X) \pt \fp\in S\}$$
where $S$ is any fixed set of places of $k$ of density 1 where $A$ has good reduction. The main theorem states that for all $X$ and all $S$ the equality $X=\overline X$ holds. A simple argument will show that in order to prove this equality, it suffices to prove that the quotient  group $\overline X/X$ is torsion free. Since $\overline X/X$ is finitely generated, it is enough to show that for all primes $\ell$ the group $(\overline X/X) \otimes\IZ_\ell$ is torsion free. But then, using Propositions \ref{Pro:InjectionOtimesEll} and \ref{Pro:H1astFundamental} this amounts to show that the group $H^1(L^M, \Tell M)$ is torsion free for a suitable 1--motive $M$. Our program consists now of establishing a general condition ensuring that $H^1_\ast(L,T)$ is torsion free for an $\ell$--adic Lie group $L$ acting on a finitely generated free $\IZ_\ell$--module $T$, and then to show that $L^M$ acting on $\Tell M$ meets this condition.
\end{par}

\vspace{4mm}
\begin{klem}\label{Lem:H1astTorsionFree}
Let $T$ be a finitely generated free $\IZ_\ell$--module, let $L$ be a Lie subgroup of $\GL(T)$ with Lie algebra $\fl$ and set $V:=T\otimes \IQ_\ell$. Suppose that
\begin{enumerate}
 \item The set $\{\pi\circ x\tq x\in\fl, \pi\in V^\ast\}$ is a linear subspace of $V^\ast$
 \item The equality $V^L=V^\fl$ holds.
\end{enumerate}
Then the group $H^1_\ast(L,T)$ is torsion free.
\end{klem}

\vspace{4mm}
\begin{para}
The proof needs some preparation. Let us introduce the following ambulant terminology: Given a finitely generated free $\IZ_\ell$--module $T$ and a Lie subgroup $L \subseteq \GL(T)$ as in the Lemma, we say that $L$ acts \emph{tightly} if the equality
$$\bigcap_{g\in L}\big(T+V^g) = T+V^L$$
where $V := T\otimes\IQ_\ell$. The inclusion $\supseteq$ always trivially holds. More generally, if $V_2$ is another $\IQ_\ell$--vector space we say that a family of linear maps $\Phi \subseteq \Hom(V, V_2)$ is \emph{tight} if the equality
\begin{equation}
\bigcap_{\varphi\in\Phi}\big(T+\ker\varphi\big) = T + \bigcap_{\varphi\in\Phi}\ker\varphi\tag{$\ast$}
\end{equation}
holds. Again the inclusion $\supseteq$ is trivial. So, $L$ acts tightly on $V$ if and only if for $V_2 =V$ the family $\{(g-1_V)\tq g\in L\}$ is tight. The following lemma shows how this is related with the torsion of $H^1_\ast(L,T)$.
\end{para}

\vspace{4mm}
\begin{lem}\label{Lem:TightEtH1ast}
Let $T$ be a finitely generated free $\IZ_\ell$--module, let $L$ be a Lie subgroup of $\GL(T)$ with Lie algebra $\fl$ and set $V:=T\otimes \IQ_\ell$. If $L$ acts tightly on $V$ then the group $H^1_\ast(L,T)$ is torsion free.
\end{lem}

\begin{proof}
Let $c:L\to T$ be a cocycle representing an element of $H^1_\ast(L,T)[\ell]$, and let us show that $c$ is a coboundary. As $\ell c$ is a coboundary, $c$ is a coboundary in $H^1(L,V)$ and there exists an element $v\in V$ such that $c(g) = gv-v$ for all $g\in L$. To say that the cohomology class of $c$ belongs to the subgroup $H^1_\ast(L,T)$ of $H^1(L,T)$ is to say that for all $g\in L$, there exists a $t_g \in T$ such that $c(g) = gt_g-t_g$. We find that
$$(g-1_V)t_g = (g-1_V)v \qquad\pt g\in L$$
or in other words $v-t_g \in\ker(g-1_V)$, that is to say $v\in T+V^g$. This is true for all $g\in L$ and since $L$ acts tightly this implies that $v=t+v_0$ for some $t\in T$ and some $v_0\in V^L$. Hence $c(g) = gt-t$ is a coboundary as needed.
\end{proof}

\vspace{4mm}
\begin{lem}\label{Lem:LinearTight}
Let $V$ and $V_2$ be $\IQ_\ell$--vector spaces with linear duals $V^\ast$ and and $V_2^\ast$ let $\Phi$ be a linear subspace of $\Hom(V,V_2)$. If the set $\Psi:=\{\pi\circ\varphi\tq \varphi\in\Phi, \pi\in V^\ast_2\}$ is a linear subspace of $V^\ast$, then $\Phi$ is tight.
\end{lem}

\begin{proof}
In ($\ast$), the inclusion $\supseteq$ holds trivially, we have to show that the inclusion $\subseteq$ holds as well. We have
$$\bigcap_{\varphi\in\Phi}\big(T+\ker\varphi\big) \subseteq \bigcap_{\psi\in\Psi}\big(T+\ker\psi\big) \qqet \bigcap_{\varphi\in\Phi}\ker\varphi = \bigcap_{\psi\in\Psi}\ker\psi$$
Hence, it is enough to show that the lemma holds in the case where $V_2 = \IQ_\ell$ and $\Phi = \Psi$. Write $W$ for the intersection of the kernels $\ker \varphi$, so that
$$W = \{v\in V \tq \varphi(v) = 0 \pt \varphi \in \Phi\} \qqet \Phi = \{\varphi\in V^\ast \tq \varphi(w)=0 \pt w\in W\}$$
Here we use that $\Phi = \Psi$ is a linear subspace of $V^\ast$. Because $T/(T\cap W)$ is torsion free the submodule $W\cap T$ is a direct factor of $T$ (every finitely generated torsion free $\IZ_\ell$--module is free, hence projective), hence we can choose a $\IZ_\ell$--basis $e_1, \ldots, e_s, \ldots, e_r$ of $T$ such that $e_1, \ldots, e_s$ make up a $\IZ_\ell$--basis of $W\cap T$. Let $v$ be an element of $V$ that is contained in $T+\ker\varphi$ for all $\varphi\in \Phi$. We can write $v$ as
$$v = \underbrace{\lambda_1e_1 + \cdots + \lambda_se_s}_{\in W} + \lambda_{s+1}e_{s+1} + \cdots + \lambda_re_r$$
where the $\lambda_i$ are scalars in $\IQ_\ell$. Taking for $\varphi$ the projection onto the $i$--th component for $s<i\leq r$ shows that $\lambda_i \in \IZ_\ell$ for $s<i\leq r$. Hence $\lambda_{s+1}e_{s+1} + \cdots + \lambda_re_r \in T$, and we find that $v \in W+T$ as required.
\end{proof}

\vspace{4mm}
\begin{proof}[Proof of Lemma \ref{Lem:H1astTorsionFree}]
Let $H$ be an open subgroup of $L$ such that the logarithm map is defined on $H$. Such a subgroup always exists, and the exponential of $\log h$ is then also defined and one has $\exp\log h = h$ for all $h\in H$ (\cite{BourLie}, Ch.II, \S8, no.4, proposition 4). The Lie algebra of $H$ is also $\fl$. Let $h$ be an element of $H$ and set $\varphi := \log h$, so that $h=\exp \varphi$. We claim that equality $V^h = \ker \varphi$ holds. On one hand if $hv=v$, then the series
$$\varphi(v) = \log h(v) = (h-1)(v)-\frac{(h-1)^2(v)}{2} + \cdots + (-1)^{n-1}\frac{(h-1)^n(v)}{n} + \cdots$$
is zero, whence $V^h \subseteq \ker\varphi $. On the other hand, if $\varphi(v) = 0$, then the series
$$h(v) = \exp \varphi (v) = 1_V(v) + \varphi(v) + \frac{\varphi^2(v)}{2} + \cdots + \frac{\varphi^n(v)}{n!} + \cdots $$
is trivial except for its first term which is $1_V(v)=v$, whence the inclusion in the other direction. The Lie algebra $\fl$ is a linear subspace of $\End V$ satisfying the hypothesis of Lemma \ref{Lem:LinearTight}. Using this lemma and the hypothesis (2) we find
$$\bigcap_{g\in L}(T+V^g) \:\:\subseteq\:\: \bigcap_{\varphi \in \fl}(T+\ker\varphi) \:\:\overset{\ref{Lem:LinearTight}}{=}\:\: T + V^\fl \:\:=\:\: T + V^L$$
hence $L$ acts tightly on $V$. By Lemma \ref{Lem:TightEtH1ast} this implies that $H^1_\ast(L,T)$ is torsion free as claimed. Mind that in the second intersection it does not matter whether we take the intersection over $\varphi\in \fl$ or $\varphi \in \log(H)$, because every element of $\fl$ is a scalar multiple of an element in $\log(H)$.
\end{proof}

\vspace{4mm}
\begin{cor}\label{Cor:TateModFixed}
Let $M = [u:Y\to A]$ be a 1--motive over a number field $k$ where $Y$ is constant and $A$ is a geometrically simple abelian variety. The group $H^1_\ast(L^M, \Tell M)$ is torsion free.
\end{cor}

\begin{proof}
We check that the two conditions of Lemma \ref{Lem:H1astTorsionFree} are satisfied. The first condition holds by Proposition \ref{Pro:LinearImageEll}. To check the second condition, we have to show that for every subgroup $N$ of $L^M$ of finite index the equality $(\Vell M)^{L^M} = (\Vell M)^N$ holds. It is enough to show this for normal subgroups, so let us fix a normal subgroup $N$ of $L^M$, and denote by $k'$ the subfield of $\overline k$ fixed by the preimage $\Gamma'$ of $N$ in $\Gamma := \Gal(\overline k|k)$. So $k'$ is a finite Galois extension of $k$, and what we have to show is that the inclusion
$$(\Tell M)^\Gamma \subseteq (\Tell M)^{\Gamma'} $$
is an equality. Indeed, by Proposition \ref{Pro:TateModFixed} and because $Y$ is constant both of these submodules of $\Tell M$ are equal to $(\ker u)\otimes\IZ_\ell$.
\end{proof}

\vspace{4mm}
\begin{proof}[Proof of the Main Theorem]
\begin{par}
We fix a geometrically simple abelian variety $A$ over a number field $k$ with algebraic closure $\overline k$. We also choose a model of $A$ over an open subscheme $U$ of $\spec\cO_k$, which we still denote by $A$, and we fix a set $S$ of closed points of $U$ of density $1$. For every subgroup $X$ of $A(k)$ we define
$$\overline X := \{P\in A(U) \tq \red_\fp(P) \in \red_\fp(X) \pt \fp\in S\}$$
Our aim is to show that for all $X \subseteq A(k)$ the equality $X=\overline X$ holds.
\end{par}

\begin{par}{\bf Claim.}
\emph{It suffices to prove that for all subgroups $X \subseteq A(k)$ the group $\overline X/X$ is torsion free.}
\end{par}

\begin{par}
Indeed, let $X$ be a subgroup of $A(k)$, and let $X'$ be any subgroup of finite index of $A(k)$ containing $X$. Because $X$ is contained in $X'$ the group $\overline X$ is contained in $\overline X'$. Moreover $X'$ is of finite index in $\overline X'$, so if $\overline X'/X'$ is torsion free we must have equality $X'=\overline X'$. Hence, as $X'$ was arbitrary, $\overline X$ is contained in every subgroup of finite index of $A(k)$ which contains $X$. This in turn implies that the equality $X=\overline X$ holds, because $A(k)$ is finitely generated.
\end{par}

\begin{par}
We now fix a subgroup $X$ of $A(k)$ and a prime number $\ell$, and we show that $\overline X/X$ contains no $\ell$--torsion, or equivalently that $(\overline X/X)\otimes \IZ_\ell$ is torsion free. Replacing $U$ by a smaller open subscheme $U'\subseteq U$ and deleting some finitely many elements from $S$ we may suppose without loss of generality that $\ell$ is invertible on $U$. Let us then choose a 1--motive $M = [u:Y\to A]$ over $U$ such that $Y$ is constant and such that $u(Y)=X$. From the propositions \ref{Pro:InjectionOtimesEll}, \ref{Pro:SerreGabber} and \ref{Pro:H1astFundamental} we get a canonical $\IZ_\ell$--linear injections
$$(\overline X/X)\otimes\IZ_\ell \xrightarrow{\:\:\ref{Pro:InjectionOtimesEll}\:\:} H^1_S(\Gamma_U, \Tell M) \xrightarrow{\:\:\ref{Pro:SerreGabber}\:\:} H^1_\ast(\Gamma_U, \Tell M)\xrightarrow{\:\:\ref{Pro:H1astFundamental}\:\:} H^1_\ast(L^M, \Tell M)$$
It is therefore enough to show that $H^1_\ast(L^M, \Tell M)$ is torsion free. But this is guaranteed by Lemma \ref{Lem:H1astTorsionFree} and the hypothesis that $A$ is geometrically simple.
\end{par}
\end{proof}

\vspace{4mm}
\begin{rem}
In the proof we only needed information on the torsion of $H^1_\ast(L^M, \Tell M)$ because of the trick that permitted us to suppose that $X$ is of finite index in $\overline X$. One can show that the group $H^1_\ast(L^M, \Tell M)$ is in fact trivial for such 1--motives. 
\end{rem}

\vspace{4mm}
\begin{par}{\bf Question 1.}
Let $G$ be a semiabelian variety over a number field $k$, let $X$ be a finitely generated subgroup of $G(k)$ and let $P\in G(U)$ be a point. Suppose that for \emph{all} finite places $v$ of $k$, the point $P$ belongs to the closure of $X$ in $G(k_v)$. Does then $P$ belong to $X$? Here, $k_v$ denotes the completion of $k$ at $v$, and we equip $G(k_v)$ with the topology induced by the topology of $k_v$. If $G$ has good reduction at $v$ and if $X$ and $P$ are integral at $v$ (so this concerns all but finitely many places) then to say that $P$ is in the closure of $X$ in $G(k_v)$ is equivalent with saying that $P$ belongs to $X$ modulo $v$, essentially by Hensel's Lemma.
\end{par}

\vspace{4mm}
\begin{par}{\bf Question 2.}
Let $A$ be an abelian variety over a number field $k$, let $X \subseteq A(k)$ be a subgroup of the group of rational points and let $P\in A(k)$ be a rational point. What can one say about the density of the set of places $\fp$ of $k$ with the property that $\red_\fp(P)$ belongs to $\red_\fp(X)$?
\end{par}

\vspace{14mm}
\bibliographystyle{amsalpha}

\providecommand{\bysame}{\leavevmode\hbox to3em{\hrulefill}\thinspace}

\providecommand{\href}[2]{#2}

\vspace{15mm}
$$ $$
\hspace{70mm}
\begin{minipage}[]{80mm}
Peter Jossen\\[1mm]
Fakult\"at f\"ur Mathematik\\
Universit\"at Regensburg\\
Universit\"atsstr. 31\\
93040 Regensburg, GERMANY\\[2mm]
{\tt peter.jossen@gmail.com} 
\end{minipage}

\end{document}